\documentclass[a4paper,UKenglish,cleveref, autoref, thm-restate]{lipics-v2021}

\usepackage{algorithm}
\usepackage[noend]{algpseudocode}

\hideLIPIcs
\nolinenumbers   

\title{The subpower membership problem of 2-nilpotent algebras} 

\author{Michael Kompatscher}{Department of Algebra, MFF, Charles University, Prague, Czechia \and \url{https://www2.karlin.mff.cuni.cz/~kompatscher/} }{ kompatscher@karlin.mff.cuni.cz}{https://orcid.org/0000-0002-0163-6604}{}

\authorrunning{M. Kompatscher} 

\Copyright{Michael Kompatscher} 



\category{} 

\relatedversion{} 

\funding{This paper was supported by the Charles University project UNCE/SCI/022.}

\acknowledgements{I would like to thank Peter Mayr for introducing me to difference clonoids, and giving several helpful comments on earlier versions of this paper.}

\newtheorem{question}[theorem]{Question}

\theoremstyle{remark}
\newtheorem{notation}[theorem]{Notation}

\newcommand{\N}{{\mathbb N}}
\newcommand{\Z}{{\mathbb Z}}

\DeclareMathOperator{\CSP}{CSP}
\DeclareMathOperator{\Con}{Con}

\DeclareMathOperator{\SMP}{SMP}
\DeclareMathOperator{\CompRep}{CompRep}

\newcommand{\alg}[1]{\mathbf{#1}}

\newcommand{\algA}{\alg{A}}
\newcommand{\algB}{\alg{B}}

\newcommand{\algL}{\alg{L}}

\newcommand{\algU}{\alg{U}}

\DeclareMathOperator{\Clo}{\mathsf{Clo}}

\DeclareMathOperator{\comP}{\mathsf{P}}
\DeclareMathOperator{\comNP}{\mathsf{NP}}

\DeclareMathOperator{\comcoNP}{\mathsf{coNP}}

\DeclareMathOperator{\comEXP}{\mathsf{EXP}}
\DeclareMathOperator{\comPS}{\mathsf{PSPACE}}

\DeclareMathOperator{\Diff}{Diff}

\DeclareMathOperator{\Sg}{Sg}

\DeclareMathOperator{\Sig}{Sig}
\DeclareMathOperator{\proj}{pr}

\EventEditors{John Q. Open and Joan R. Access}
\EventNoEds{2}
\EventLongTitle{42nd Conference on Very Important Topics (CVIT 2016)}
\EventShortTitle{CVIT 2016}
\EventAcronym{CVIT}
\EventYear{2016}
\EventDate{December 24--27, 2016}
\EventLocation{Little Whinging, United Kingdom}
\EventLogo{}
\SeriesVolume{42}
\ArticleNo{23}

\begin{document}

\maketitle

\begin{abstract}
The subpower membership problem $\SMP(\algA)$ of a finite algebraic structure $\algA$ asks whether a given partial function from $A^k$ to $A$ can be interpolated by a term operation of $\algA$, or not. While this problem can be EXPTIME-complete in general, Willard asked whether it is always solvable in polynomial time if $\algA$ is a Mal'tsev algebras. In particular, this includes many important structures studied in abstract algebra, such as groups, quasigroups, rings, Boolean algebras. In this paper we give an affirmative answer to Willard's question for a big class of 2-nilpotent Mal'tsev algebras. We furthermore develop tools that might be essential in answering the question for general nilpotent Mal'tsev algebras in the future.
\end{abstract}

\section{Introduction}
It is a recurring and well-studied problem in algebra to describe the closure of a given list of elements under some algebraic operations (let us only mention the affine and linear closure of a list of vectors, or the ideal generated by a list of polynomials). But also in a computational context, this problem has a rich history, appearing in many areas of computer science. In its formulation as \emph{subalgebra membership problem}, the task is to decide whether a given finite list of elements of an algebraic structure generates another element or not. 

Depending on the algebraic structures studied, a variety of different problems emerges. One of the most well-known examples is the \emph{subgroup membership problem}, in which the task is to decide, if for a given set of permutations $\alpha_1,\ldots,\alpha_n$ on a finite set $X$, another permutation $\beta$ belongs to the subgroup generated by $\alpha_1,\ldots,\alpha_n$ in $S_X$. This problem can be solved in polynomial-time by the famous Schreier-Sims algorithm \cite{sims-SSalgorithm}, whose runtime was analysed in \cite{FHL-permuationgroups} and \cite{knuth-SS}. The existence of such efficient algorithms is however not always guaranteed: if the symmetric group $S_X$ is for instance replaced by the full transformation semigroup on $X$, the corresponding membership problem is $\comPS$-complete \cite{kozen-lowerbounds}.

A common feature of many algorithms for the subalgebra membership problem is to generate canonical generating sets of some sorts (such as computing the basis of a vector space via Gaussian elimination, or computing a Gr\"obner basis via Buchberger's algorithm to solve the ideal membership problem \cite{buchberger-phd}). But, in general, this is where the similarities end - depending on the algebraic structure, and the encoding of the input, the problem can range over a wide range of complexities, and have applications in vastly different areas such as cryptography \cite{SU-Thompsonsgroup, SZ-sgmcrypto}, computer algebra \cite{buchberger-phd, MM-IMPhard}, or proof complexity \cite{kozen-lowerbounds, kozen-complexityalgebras}.


In this paper, we study a version of the subalgebra membership problem that is called the \emph{subpower membership problem}. For a fixed, finite algebraic structure $\algA$ (henceforth also just called an \emph{algebra}) its subpower membership problem $\SMP(\algA)$ is the problem of deciding if a given tuple $\mathbf b \in \algA^k$ is in the subalgebra of $\algA^k$ generated by some other input tuples $\mathbf a_1,\ldots,\mathbf a_n \in \algA^k$ (here $n$ and $k$ are not fixed, but part of the input). This is equivalent to checking, whether the $n$-ary partial function that maps $\mathbf a_1,\ldots,\mathbf a_n$ component-wise to $\mathbf b$ can be interpolated by a term function of $\algA$. For example, if $p$ is a prime, $\SMP(\Z_p)$ is the problem of checking whether some vector $\mathbf b \in \Z_p^k$ is in the linear closure of $\mathbf a_1,\ldots,\mathbf a_n \in \Z_p^k$; this can easily be solved by Gaussian elimination. More general, for any finite group $\mathbf G$,  $\SMP(\mathbf G)$ can be solved in polynomial time by a version of the Schreier-Sims algorithm \cite{willard-talk}.

Besides being a natural problem in algebra, the subpower membership problem found some applications in some learning algorithms \cite{BD-Maltsev, DJ-learability, IMMVW-fewsubpowers}.  Moreover, an efficient algorithm for $\SMP(\algA)$ implies that it is also feasible to represent the relations invariant by some generating set of tuples. It was in particular remarked (see e.g.~\cite{BMS-SMP}), that a polynomial-time algorithm for $\SMP(\algA)$ would allow to define infinitary constraint satisfaction problems, in which the constraint relations are given by some generating tuples (with respect to $\algA$). This infinitary version of $\CSP$s has the benefit that most of the algebraic machinery to $\CSP$s (see e.g. \cite{BKW-polymorphisms}) still applies.

Exhaustively generating the whole subalgebra generated by $\mathbf a_1,\ldots,\mathbf a_n$ in $ \algA^k$ gives an exponential time algorithm for $\SMP(\algA)$. And, in general, we cannot expect to do better: In \cite{kozik-SMPhard} Kozik constructed a finite algebra $\algA$ for which $\SMP(\algA)$ is $\comEXP$-complete. Even semigroups can have $\comPS$-complete subpower membership problem \cite{BCMS-SMPsemigroup}.

However, for so called \emph{Mal'tsev algebras},  better lower bounds are known.  Mal'tsev algebras are algebras defined by having a \emph{Mal'tsev term} $m$, i.e. a term satisfying the identities $y = m(x,x,y) = m(y,x,x)$ for all $x,y$. Mal'tsev algebras lie at the intersection of many areas of mathematics: they include algebraic structures of ubiquitous importance (groups, fields, vector spaces), but also appear in logic (Boolean algebras, Heyting algebras), commutative algebra (rings, modules, $K$-algebras), and non-associative mathematics (quasigroups, loops). Mayr showed in \cite{mayr-SMP} that the subpower membership problem of every Mal'tsev algebra is in $\comNP$. His proof is based on the fact that every subalgebra $\mathbf R \leq \algA^n$ has a small generating set, which generates every element of $\mathbf R$ in a canonical way (a so-called \emph{compact representation}). Thus, to solve the subpower membership problem, one can ``guess'' a compact representation of the subalgebra generated by $\mathbf a_1,\ldots,\mathbf a_k$, and then check in polynomial time if it generates $\mathbf b$.  If such a compact representation can be moreover found in \emph{deterministic} polynomial time, then $\SMP(\algA)$ is in $\comP$; this is, in fact, the dominant strategy to prove tractability.

So far, the existence of such polynomial time algorithms was verified for groups and rings \cite{willard-talk,  FHL-permuationgroups}, supernilpotent algebras \cite{mayr-SMP}, and algebras that generate residually finite varieties \cite{BMS-SMP}.  On the other hand, no examples of $\comNP$-hard or intermediate complexity are known. This leads to the question whether $\SMP(\algA) \in \comP$ for \emph{all} finite Mal'tsev algebras $\algA$ \cite{willard-talk}.  On a broader scale, this question was also posed for algebras with \emph{few subpowers} \cite[Question 8]{IMMVW-fewsubpowers}.

An elementary class of Mal'tsev algebras, for which the question still remains open, are \emph{nilpotent} algebra. In fact, they can also be seen as an important stepping stone in answering \cite[Question 8]{IMMVW-fewsubpowers}, as nilpotent Mal'tsev algebras coincide with nilpotent algebras with few subpowers. Generalizing the concept of nilpotent groups, nilpotent algebras are defined by having a central series of congruences. While they have several nice structural properties, in general nilpotent algebras do not satisfy the two finiteness conditions mentioned above (supernilpotence, residual finiteness), thus they are a natural starting point when trying to generalize known tractability results.  But even for 2-nilpotent algebras not much is known, then all polynomial-time algorithms were only constructed by ad-hoc arguments for concrete examples (such as Vaughan-Lee's 12-element loop \cite{mayr-loop}).

The first contribution of this paper is to prove that all 2-nilpotent algebras of size $p\cdot q$ for two primes $p \neq q$ have a tractable subpower membership problem. In fact, we prove an even stronger result in Theorem \ref{theorem:main}: $\SMP(\algA)$ is in $\comP$, whenever $\algA$ has a central series $0_\algA < \rho < 1_\algA$ such that $|\algA/\rho| = p$ is a prime, and the blocks of $\rho$ have size coprime to $p$.

While this is still a relatively restricted class of nilpotent algebras, our methods have the potential to generalize to all 2-nilpotent Mal'tsev algebras and beyond. Thus, our newly developed tools to analyze $\SMP$ can be regarded as the second main contribution. More specifically, in Theorem \ref{theorem:SMP} we show that whenever $\algL \otimes \algU$ is a \emph{wreath product} (see Section \ref{sect:diffclon}), such that $\algU$ is supernilpotent, then $\SMP(\algL \otimes \algU)$ reduces to $\SMP(\algL \times \algU)$ (which is polynomial-time solvable by \cite{mayr-SMP}) and a version of the subpower membership problem for a multi-sorted algebraic object called a \emph{clonoid} from $\algU$ to $\algL$. This reduction in particular applies to all 2-nilpotent algebras; an analysis of clonoids between affine algebras then leads to Theorem \ref{theorem:main}. If, in future research, we could get rid of the condition of $\algU$ being supernilpotent, this would provide a strong tool in studying general Mal'tsev algebras, as every Mal'tsev algebra with non-trivial center can be decomposed into a wreath product.

Our paper is structured as follows: Section \ref{sect:prelim} contains preliminaries and some background on universal algebra. In Section \ref{sect:diffclon} we discuss how Mal'tsev algebras with non-trivial center can be represented by a wreath product and we introduce the concept of \emph{difference clonoid} of such a representation. In Section \ref{sect:reduction} we discuss some situations, in which the subpower membership problem of a wreath product can be reduced to the membership problem of the corresponding difference clonoid. In particular, we prove Theorem \ref{theorem:SMP}. Section \ref{section:clonoids} contains an analysis of clonoids between $\Z_p$ and coprime Abelian groups, which then leads to the proof of our main result, Theorem \ref{theorem:main}. In Section \ref{sect:discussion}
we discuss some possible directions for future research.


\section{Preliminaries} \label{sect:prelim}
In the following, we are going to discuss some necessary notions from universal algebra. For more general background we refer to the textbooks \cite{bergman-book, BS-book}. For background on commutator theory we refer to \cite{FM-commutator-theory} and \cite{AM-commutator}. For an introduction to Malt'sev algebras and compact representations we refer to \cite[Chapters 1.7-1.9]{brady-CSPnotes}.

In this paper, we are going to denote tuples by lower case bold letters, e.g. $\mathbf a \in A^k$. In order to avoid double indexing in some situations, we are going to use the notation $\mathbf a(i)$ to denote the $i$-th entry of $\mathbf a$, i.e. $\mathbf a = (\mathbf a(1), \mathbf a(2), \ldots, \mathbf a(k))$. However, otherwise we are going to follow standard notation as used e.g. in \cite{bergman-book}.

\subsection{Basic notions for general algebras}
An \emph{algebra} $\algA = (A;(f_i^{\algA})_{i \in I})$ is a first-order structure in a purely functional language $(f_i)_{i \in I}$ (where each symbol $f_i$ has an associated \emph{arity}). We say $\algA$ is finite if its domain $A$ is finite.  A \emph{subalgebra} $\algB = (B;(f_i^{\algB})_{i \in I})$ of an algebra $\algA = (A;(f_i^{\algA})_{i \in I})$ (denoted $\algB \leq \algA$) is an algebra obtained by restricting all \emph{basic operations} $f_i^{\algA}$ to a subset $B \subseteq A$ that is invariant under all $f_i^{\algA}$'s. The subalgebra generated by a list of elements $a_1,\ldots,a_n$, denoted by $\Sg_\algA(a_1,\ldots,a_n)$ is the smallest subalgebra of $\algA$ that contains $a_1,\ldots,a_n$. The \emph{product} $\prod_{i \in I} \algA_i$ of a family of algebras $(\algA_i)_{i \in I}$ in the same language is defined as the algebra with domain $\prod_{i\in I} A_i$, whose basic operations are defined coordinate-wise. The power $\algA^n$ is the product of $n$-many copies of $\algA$. Subalgebras of (finite) powers of $\algA$ are sometimes also called \emph{subpowers} of $\algA$, which motivates the name ``subpower membership problem''. So, formally the subpower membership problem of $\algA$ can be stated as follows:\\

\noindent $\SMP(\algA)$\\
\textsc{Input:} $\mathbf b, \mathbf a_1,\ldots,\mathbf a_n \in \algA^k$ for some $n,k\in\N$\\
\textsc{Question:} Is $\mathbf b \in \Sg_{\algA^k}(\mathbf a_1,\ldots,\mathbf a_n)$?\\

Note that the subpowers of $\algA$ are exactly the relations on $A$ that are invariant under $\algA$. A \emph{congruence} $\alpha$ of $\algA$ is an equivalence relation on $A$ that is invariant under $\algA$. We write $\Con(\algA)$ for the lattice of all congruence of $\algA$. We denote the minimal and maximal element of this lattice by $0_\algA = \{(x,x) \mid x \in A\}$ and $1_\algA = \{(x,y) \mid x,y \in A\}$. For every congruence $\alpha \in \Con(\algA)$, one can form a quotient algebras $\algA/\alpha$ in the natural way.

The \emph{term operations} of an algebra $\algA$ are all finitary operations that can be defined by a composition of basic operations of $\algA$. Two standard ways to represent them is by terms or circuits in the language of $\algA$. For a term or circuit $t(x_1,\ldots,x_n)$ in the language of $\algA$, we write $t^{\algA}(x_1,\ldots,x_n)$ for the induced term operation on $A$.  Occasionally, if it is clear from the context, we are not going to distinguish between a term/circuit and the corresponding term operation. The term operations of an algebra $\algA$ are closed under composition and contain all projections, therefore they form an algebraic object called a \emph{clone}. For short, we denote this \emph{term clone} of an algebra $\algA$ by $\Clo(\algA)$. Note that $\Sg_{\algA^k}(\mathbf a_1,\ldots,\mathbf a_n) = \{t(\mathbf a_1,\ldots,\mathbf a_n) \mid t \in \Clo(\algA) \}$.

We call a ternary operation $m^{\algA}(x,y,z) \in \Clo(\algA)$ a \emph{Mal'tsev term} if it satisfies the identities $m^{\algA}(y,x,x) = m^{\algA}(x,x,y) = y$ for all $x,y\in A$, and call $\algA$ a \emph{Mal'tsev algebra} if it has a Mal'tsev term. For instance, every group is a Mal'tsev algebra with Mal'tsev term $m(x,y,z)=xy^{-1}z$. Mal'tsev terms are a classic topic of study in universal algebra (see e.g. \cite[Chapter 7]{bergman-book}), and are in particular known to characterize congruence permutable varieties.



\subsection{Clonoids}
We are also going to rely on a multi-sorted generalisation of clones, so-called \emph{clonoids} that were first introduced in \cite{AM-clonoids} (in a slightly less general way). For a set of operations between two sets $\mathcal C \subseteq \{f \colon A^n \to B \mid n\in \mathbb N\}$, and $k \in \N$ let us write $\mathcal C^{(k)} = \{f \colon A^k \to B \mid f\in \mathcal C\}$ for the subset of $k$-ary functions. Then, for two algebras $\algA = (A,(f_i)_{i\in I})$, $\algB = (B,(g_j)_{j\in J})$ (in possibly different domains and languages), a set $\mathcal C \subseteq \{f \colon A^n \to B \mid n \in \mathbb N\}$ is called a \emph{clonoid from $\algA$ to $\algB$}, or \emph{$(\algA,\algB)$-clonoid}, if it is closed under composition with term operations of $\algA$ from the inside, and $\algB$ from the outside, i.e.: $\forall n,k\in \N$
\begin{bracketenumerate}
\item $f \in \mathcal C^{(n)},  t_1,\ldots, t_n \in \Clo(\algA)^{(k)} \Rightarrow f \circ (t_1,\ldots, t_n) \in \mathcal C^{(k)}$ 
\item $s \in \Clo(\algB)^{(n)}, f_1,\ldots, f_n \in \mathcal C^{(k)} \Rightarrow s \circ (f_1,\ldots, f_n) \in \mathcal C^{(k)}$. 
\end{bracketenumerate}

\subsection{Commutator theory}
Commutator theory is the subfield of universal algebra that tries to generalise notions such as central subgroups, nilpotence, or solvability from group theory to general algebras. The most commonly used framework is based on so-called term-conditions, which we outline in the following.

Let $\algA$ be an algebra. For congruences $\alpha,\beta,\gamma \in \Con(\algA)$ we say that \emph{$\alpha$ centralized $\beta$ modulo $\gamma$} (and write $C(\alpha,\beta;\gamma)$) if and only if for all $p(\mathbf x, \mathbf y) \in \Clo(\algA)$, and all tuples $\mathbf a, \mathbf b \in A^n$, $\mathbf c, \mathbf d \in A^m$, such that $a_i \sim_{\alpha} b_i$ for $i=1,\ldots,n$ and $c_j \sim_{\beta} d_j$ for $j=1,\ldots,m$, the implication
\begin{align*}
p(\mathbf a, \mathbf c) \sim_{\gamma} p(\mathbf a, \mathbf d) \Rightarrow p(\mathbf b, \mathbf c) \sim_{\gamma} p(\mathbf b, \mathbf d)
\end{align*}
holds.
A congruence $\alpha$ is called \emph{central} if $C(\alpha,0_{\algA};1_{\algA})$ holds. The \emph{center} is the biggest central congruence. An algebra $\algA$ is called \emph{$n$-nilpotent} if there is a \emph{central series of length $n$}, i.e. a series of congruences $0_{\algA} = \alpha_0 \leq \alpha_1 \leq \cdots \leq \alpha_n = 1_{\algA}$, such that $C(\alpha_{i+1},1_{\algA}; \alpha_i)$ for $i = 0,\ldots,n-1$. An algebra $\algA$ is called \emph{Abelian}, if it is $1$-nilpotent, i.e. $C(1_{\algA},1_{\algA}; 0_{\algA})$ holds.

We are, however, not going to work directly with these definitions. There is a rich structural theory in the special case of Mal'tsev algebras (and, more general, in congruence modular varieties \cite{FM-commutator-theory}) that gives us very useful characterizations of many commutator theoretical properties. 

By a result of Herrmann \cite{herrmann-affine}, a Mal'tsev algebra $\algA$ is Abelian if and only if it is \emph{affine}, i.e.  all of its term operations are affine combination $\sum_{i=1}^n \alpha_i x_i + c$ over some module; in particular the Mal'tsev term is then equal to $x-y+z$. More generally, we are going to a result of Freese and McKenzie \cite{FM-commutator-theory} that states that a Mal'tsev algebra $\algA$ with a central congruence $\rho$ can always be written as a \emph{wreath product} $\algL \otimes \algU$, such that $\algL$ is affine and $\algU = \algA/\rho$. We are going to discuss such wreath product representations in Section \ref{sect:diffclon}.

Lastly, we want to mention that the definition of the relation $C$ naturally generalizes to higher arities $C(\alpha_1, \ldots,\alpha_n, \beta; \gamma)$. This notion was first introduced by Bulatov; we refer to \cite{FM-commutator-theory} and \cite{AM-commutator} to more background on \emph{higher commutators}. In particular, an algebra is called \emph{$k$-supernilpotent} if $C(1_{\algA}, \ldots,1_{\algA}; 0_{\algA})$, where $1_{\algA}$ appears $k+1$ times. There are several known characterizations of supernilpotent Mal'tsev algebras. We are mainly going to use the following:

\begin{theorem}[Proposition 7.7. in \cite{AM-commutator}] \label{theorem:supernilpotent}
Let $\algA$ be a $k$-supernilpotent Mal'tsev algebra, $0 \in A$ a constant and $t,s$ two $n$-ary terms in the language of $\algA$.Then $t^{\algA} = s^{\algA}$ if and only if they are equal on all tuples from the set $S = \{ \mathbf a \in A^n \mid |\{i : \mathbf a(i) \neq 0\}| \leq k \}$.  (In fact, this is a characterization of $k$-supernilpotence for Mal'tsev algebras.)
\end{theorem}

\subsection{Compact representations and SMP}
For any subset $R \subseteq A^n$, we define its \emph{signature} $\Sig(R)$ to be the set of all triples $(i,a,b) \in \{1,\ldots,n\} \times A^2$, such that there are $\mathbf t_a, \mathbf t_b \in R$ that agree on the first $i-1$ coordinates, and $t_a(i) = a$ and $t_b(i) = b$; we then also say that $\mathbf t_a, \mathbf t_b$ are \emph{witnesses} for $(i,a,b) \in \Sig(R)$.

If $\algA$ is a Mal'tsev algebra, and $\mathbf R\leq \algA^n$, then it is known that $\mathbf R$ is already generated by every subset $S \subseteq \mathbf R$ with $\Sig(S) = \Sig(\mathbf R)$ \cite[Theorem 1.8.2.]{brady-CSPnotes}. In fact, $\mathbf R$ is then equal to the closure of $S$ under the Mal'tsev operation $m$ alone, and a tuple $\mathbf a$ is in $\mathbf R$ iff can be written as $m(\ldots m(\mathbf a_1,\mathbf b_2,\mathbf a_2), \ldots, \mathbf b_n,\mathbf a_n)$, for some $\mathbf a_i, \mathbf b_i \in S$. For given $\mathbf a \in \mathbf R$ such elements $\mathbf a_i, \mathbf b_i \in S$ can be found polynomial time in $|S|$, by picking $\mathbf a_1$ such that 
$\mathbf a_1(1) = \mathbf a(1)$, and $\mathbf a_i, \mathbf b_i \in S$ as witnesses for $m(\ldots m(\mathbf a_1,\mathbf b_2,\mathbf a_2), \ldots, \mathbf b_{i-1},\mathbf a_{i-1}))(i)$ and $a(i)$ at position $i$.

A \emph{compact representation} of $\mathbf R \leq \algA^n$ is a subset $S \subset \mathbf R$ with $\Sig(S) = \Sig(\mathbf R)$ and $|S| \leq 2|\Sig(\mathbf R)| \leq 2n|A|^2$. So, informally speaking, compact representations are small generating sets of $\mathbf R$ with the same signature. It is not hard to see that compact representations always exist.  Generalizations of compact representations exist also for relations on different domains ($\mathbf R \leq \algA_1\times \cdots \times \algA_n$), and relations invariant under algebras with few subpowers, we refer to \cite[Chapter 2]{brady-CSPnotes} for more background.

By the above, $\SMP(\algA)$ reduces in polynomial time to the problem of finding a compact representation of $\Sg_{\algA^k}(\mathbf a_1,\ldots,\mathbf a_n)$ for some input tuples $\mathbf a_1,\ldots,\mathbf a_n \in A^k$.  We are going to denote this problem by $\CompRep(\algA)$.  Conversely, it was shown in \cite{BMS-SMP} that finding a compact representations has a polynomial Turing reduction to $\SMP(\algA)$.  Note further that, to solve $\CompRep(\algA)$ it is already enough to find a subset $S \subseteq R$ with $\Sig(S) = \Sig(R)$ of polynomial size, since such a set $S$ can then be thinned out to a compact representation.

Let us call a set of pairs $\{ (\mathbf c, p_{\mathbf c}) \mid \mathbf c \in S\}$ an \emph{enumerated compact representation} of $\Sg_{\algA^k}(\mathbf a_1,\ldots,\mathbf a_n)$, if $S$ is a compact representation of $\Sg_{\algA^k}(\mathbf a_1,\ldots,\mathbf a_n)$, and every $p_{\mathbf c}$ is a circuit in the language of $\algA$ of polynomial size (in $n$ and $k$), such that $p_{\mathbf c}(\mathbf a_1,\ldots,\mathbf a_n) = \mathbf c$. Enumerated compact representations were already (implicitly) used in several proofs.  In \cite[Theorem 4.13.]{BMS-SMP} it was shown that, for algebras with few subpowers, enumerated compact representations always exist; this was used to prove that $\SMP(\algA) \in \comcoNP$.  Moreover, all of the known polynomial time algorithms for $\CompRep(\algA)$, in fact, compute enumerated compact representations. We are in particular going to need the following result that follows from \cite{mayr-SMP}: 

\begin{theorem}[\cite{mayr-SMP}] \label{theorem:SMPsupernilpotent}
Let $\algA$ be a finite supernilpotent Mal'tsev algebra. Then, there is a polynomial time algorithm that computes an enumerated compact representations of $\Sg_{\algA^k}(\mathbf a_1,\ldots,\mathbf a_n)$, for given $\mathbf a_1,\ldots,\mathbf a_n \in A^k$.
\end{theorem}

Theorem \ref{theorem:SMPsupernilpotent} can be seen as a generalization of Gaussian elimination from affine to supernilpotent algebras. We remark that Theorem \ref{theorem:SMPsupernilpotent}, although not explicitly stated as such in \cite{mayr-SMP}, follows directly from the algorithm computing a \emph{group representations} $(T_1,T_2,\ldots,T_k)$ of $\Sg_{\algA^k}(\mathbf a_1,\ldots,\mathbf a_n)$ and the fact that for such a group representation, there is a constant $q$ such that $T = (T_1+q\cdot T_2+\cdots +q\cdot T_k)$ has the same signature as $\Sg_{\algA^k}(\mathbf a_1,\ldots,\mathbf a_n)$ (see Lemma 3.1. in \cite{mayr-SMP}). Thus, $T$ together with its defining circuits forms an enumerated compact representation of $\Sg_{\algA^k}(\mathbf a_1,\ldots,\mathbf a_n)$.

We are furthermore going to use that there is an algorithm that allows us to fix some values of a relation given by enumerated compact representation:

\begin{lemma} \label{lemma:fixvalues}
Let $\algA$ be a Mal'tsev algebra. Then, there is a polynomial-time algorithm \texttt{Fix-values}$(R,a_1,\ldots,a_m)$ that, for a given compact representation $R$ of $\mathbf R = \Sg_{\algA^k}(X)$, and constants $a_1,\ldots,a_m \in A$, returns a compact representation $R'$ of $\{ \mathbf x \in \mathbf R \mid \mathbf x(1) = a_1, \ldots, \mathbf x(m) = a_m\}$ (or $\emptyset$ if the relation is empty). If $R$ is moreover enumerated then \texttt{Fix-values} also computes polynomial size circuits defining the elements of $R'$ from $X$.
\end{lemma}

The existence of such a \texttt{Fix-values} algorithm for compact representation is a well-known result (\cite{BD-Maltsev}, see also \cite[Algorithm 5]{brady-CSPnotes}); the additional statement about \emph{enumerated} compact representation follows easily from bookkeeping the defining circuits. We prove Lemma \ref{lemma:fixvalues} in Appendix \ref{appendix:fixval}.

\section{Wreath products and difference clonoids} \label{sect:diffclon}
In this section, we discuss how to represent Mal'tsev algebras with non-trivial center by a so-called \emph{wreath product} $\algL \otimes \algU$, and associate to it its \emph{difference clonoid}, which gives us a measure on how far it is from being the direct product $\algL\times \algU$.


 \begin{definition}
 Let $\algU = (U,(f^{\algU})_{f\in F})$ and $\algL = (L,(f^{\algL})_{f\in F})$ be two algebras in the same language $F$, such that $\algL$ is affine. Furthermore, let $0 \in L$ and $T = (\hat f)_{f\in F}$ be a family of operations $\hat f \colon U^n\to L$, for each $f\in F$ of arity $n$. Then we define the \emph{wreath product} $\algL \otimes^{T,0} \algU$ as the algebra $(L\times U,(f^{\algL \otimes^T \algU})_{f\in F})$ with basic operations
 $$f^{\algL \otimes^{T,0} \algU}((l_1,u_1),\ldots,(l_n,u_n)) = (f^{\algL}(l_1,\ldots,l_n)+\hat f(u_1,\ldots,u_n),f^{\algU}(u_1,\ldots,u_n)),$$
(where $+$ is the addition on $\algL$ with respect to neutral element $0$). For simplicity, we are going to write $\algL \otimes \algU$, if $T$ and $0$ are clear from the context.
\end{definition}

The name \emph{wreath product} refers to the fact that this is a special case of VanderWerf's wreath products \cite{vanderwerf-wreathproduct}. We remark that recently also alternative names for $\algL \otimes \algU$ were suggested, such as \emph{central extension} (by Mayr) and \emph{semidirect product} (by Zhuk). By a result of Freese and McKenzie we can represent Mal'tsev algebras with non-trivial centers as wreath products: 

\begin{theorem}[Proposition 7.1.  in \cite{FM-commutator-theory}] \label{theorem:wreath}
Let $\algA$ be a Mal'tsev algebra with a central congruence $\alpha$, and let $\algU =\algA/\alpha$.  Then there is an affine algebra $\algL$,  an element $0 \in L$ and a set of operations $T$, such that $\algA \cong \algL \otimes^{T,0} \algU$.
\end{theorem}
 
Note that, for a fixed quotient $\algU = \algA/\alpha$, there is still some freedom in how to choose the operations $f^\algL$ of $\algL$, and the operations $\hat f \colon U^n \to L$ in $T$ (by adding/subtracting constants). To get rid of this problem, we are from now on always going to assume that $\algL$ preserves $0$, i.e. $f^{\algL}(0,0,\ldots,0) = 0$ for all $f\in F$.   It is then easy to observe that wreath products $\algL \otimes^{0,T} \algU$ behaves nicely with respect to the direct product $\algL \times \algU$ in the same language:

\begin{observation} \label{observation:companion}
Let $\algA$ be a Mal'tsev algebra with wreath product representation $\algA = \algL \otimes^{0,T} \algU$. Then $t^{\algA} = s^{\algA} \Rightarrow t^{\algL\times \algU} = s^{\algL\times \algU}$.
\end{observation}

\begin{proof}
Note that, for every term $t$ in the language of $\algA$:
$$t^{\algA}((l_1,u_1),\ldots,(l_n,u_n)) = (t^{\algL}(l_1,\ldots,l_n)+\hat t(u_1,\ldots,u_n),t^{\algU}(u_1,\ldots,u_n)),$$
for some $\hat t \colon U^n \to L$ (this can be shown by induction over the height of the term tree). Clearly $t^{\algA} = s^{\algA}$ implies $t^{\algU}=s^{\algU}$, and $t^{\algL} - s^{\algL} = c$, $\hat t - \hat s = -c$ for some constant $c\in L$. Since, by our assumptions, the operations of $\algL$ preserve $0$, we get $t^{\algL} = s^{\algL}$ and $\hat t = \hat s$. Thus $t^{\algL\times \algU} = s^{\algL\times \algU}$.
\end{proof}

In other terminology, the map $t^{\algA} \mapsto t^{\algL \times \algU}$ is a surjective \emph{clone homomorphism} from $\Clo(\algA)$ to $\Clo(\algL\times\algU)$, i.e. a map that preserves arities, projections and compositions. The converse of Observation \ref{observation:companion} does however not hold, since this map is usually not injective. We define the \emph{difference clonoid} $\Diff_0(\algA)$ as the kernel of the clone homomorphisms in the following sense:

\begin{definition} \label{definition:diffclonoid}
Let $\algA = \algL \otimes^{0,T} \algU$ be a Mal'tsev algebra given as a wreath product. 
\begin{bracketenumerate}
    \item We define the equivalence relation $\sim$ on $\Clo(\algA)$ by $$t^{\algA}\sim s^{\algA} :\Leftrightarrow t^{\algL\times \algU} = s^{\algL\times \algU}$$
    \item the \emph{difference clonoid $\Diff_0(\algA)$} is defined as the set of all operation $\hat r \colon U^n \to L$, such that there are $t^{\algA} \sim s^{\algA} \in \Clo(\algA)$ with:
\end{bracketenumerate} 
\begin{align}
\label{eq7} t^{\algA}((l_1,u_1),\ldots,(l_n,u_n)) &= (t^{\algL}(\mathbf l) + \hat t(\mathbf u), t^{\algU}(\mathbf u))\\
\label{eq8} s^{\algA}((l_1,u_1),\ldots,(l_n,u_n)) &= (t^{\algL}(\mathbf l) + \hat t(\mathbf u) + \hat r(\mathbf u), t^{\algU}(\mathbf u))
\end{align}
\end{definition}

\begin{notation} \label{notation:diff}
In the following,  we will stick to the following convention: Function symbols with a hat will always denote operations from some power of $U$ to $L$. For operations $t,s\colon A^n \to A$, and $\hat r \colon U^n \to L$ such as in (\ref{eq7}) and (\ref{eq8}) we are slightly going to abuse notation, and write $s = t + \hat r$ and $\hat r = (s-t)$.
\end{notation}

We next show that $\Diff_0(\algA)$ is indeed a clonoid from $\algU$ to $\algL$ (extended by the constant $0$).

\begin{lemma} \label{lemma:difference} \,
Let $\algA = \algL \otimes^{0,T} \algU$ be a Mal'tsev algebra given as wreath product. Then:
\begin{bracketenumerate}
\item For all $t \in \Clo(\algA)$, $\hat r \in \Diff_0(\algA)$ also $t+\hat r \in \Clo(\algA)$,
\item $\Diff_0(\algA)$ is a $(\algU,(\algL,0))$-clonoid.
\end{bracketenumerate}
\end{lemma}

\begin{proof}
To prove (1), let $t \in \Clo(\algA)$ and $\hat r \in \Diff_0(\algA)$. By definition of the difference clonoid, $\hat r = s_1 -s_2$ for two terms $s_1, s_2 \in \Clo(\algA)$, with $s_1 \sim s_2$. In particular, $s_1^{\algU} = s_2^{\algU}$. For any Mal'tsev term $m \in \Clo(\algA)$, necessarily $\hat m(u,u,v) = \hat m(v,u,u) = 0$ holds. This implies that
$$t+\hat r = m(t,s_2,s_1) \in\Clo(\algA).$$

We next prove (2). So we only need to verify that $\Diff_0(\algA)$  is closed under composition with $\Clo(\algU)$ (from the inside), respectively $\Clo((\algL,0))$ (from the outside).

To see that $\Diff_0(\algA)$ is closed under $(\algL,0)$, note that $0 \in \Diff_0(\algA)$, as $t-t = 0$, for every term $t\in \Clo(\algA)$. Further $(\algL,0)$ is closed under $+$; for this, let $\hat r_1,\hat r_2 \in \Diff_0(\algA)$. By (1), we know that $t+\hat r_1 \in \Clo(\algA)$, for some term $t\in \Clo(\algA)$. Again, by (1) also $(t+\hat r_1) + \hat r_2)) \in \Clo(\algA)$, which shows that $\hat r_1+\hat r_2 \in \Diff_0(\algA)$. For all unary $e^{\algL} \in \Clo(\algL)$, and $t\sim s$ with $\hat r = t-s$, note that $e^{\algA} t-e^{\algA} s = e^{\algL} \circ \hat r \in \Diff_0(\algA)$. Since $\algL$ is affine, $\Clo(\algL,0)$ is generated by $+$ and its unary terms, thus $\Diff_0(\algA)$ is closed under $(\algL,0)$.

To see that $\Diff_0(\algA)$ is closed under $\algU$ from the inside, simply notice that $t(x_1,\ldots,x_n) \sim s(x_1,\ldots,x_n)$ implies $t(f_1(\mathbf x),\ldots,f_n(\mathbf x)) \sim s(f_1(\mathbf x),\ldots,f_n(\mathbf x))$, for all terms $f_1,\ldots,f_n$. If $\hat r = t^{\algA}-s^{\algA}$, then $\hat r \circ (f_1^{\algU},\ldots,f_n^{\algU}) = t\circ (f_1^{\algU},\ldots,f_n^{\algU})- s\circ (f_1^{\algU},\ldots,f_n^{\algU}) \in \Diff_0(\algA)$.
\end{proof}

We remark that the choice of the constant $0 \in L$ is not relevant in this construction: since for every $c \in L$ the map $\hat r \mapsto \hat r + c$ is an isomorphism between the $(\algU,(\algL,0))$-clonoid $\Diff_0(\algA)$ and the $(\algU,(\algL',c))$-clonoid $\Diff_c(\algA)$ (where $f^{\algL'}(\mathbf l) = f^{\algL}(\mathbf l-(c,c\ldots,c))+c$).

Our goal in the next section is to reduce the subpower membership problem to a version of the subpower membership problem for the difference clonoid in which we ask for membership of a tuple $\mathbf l \in L^k$ in the subalgebra of $\algL$ given by the image of $\mathbf u_1,\ldots,\mathbf u_n \in U^k$ under the clonoid.  In fact,  it will be more convenient for us to ask for a compact representation, that's why we define the following problem, for a clonoid $\mathcal C$ from $\algU$ to $\algL$.\\

\noindent $\CompRep(\mathcal C)$:\\
\textsc{Input:} A list of tuples $\mathbf u_1,\ldots,\mathbf u_n \in U^k$.\\
\textsc{Output:} A compact representation of $\mathcal C(\mathbf u_1,\ldots,\mathbf u_n) = \{f(\mathbf u_1,\ldots,\mathbf u_n) \mid f \in \mathcal C \} \leq \algL^{k} $\\

In the case of the difference clonoid $\mathcal C = \Diff_0(\algA)$ the image algebra $\algL$ is affine and contains a constant $0$. Thus then this problem is then equivalent to finding generating set of $\mathcal C(\mathbf u_1,\ldots,\mathbf u_n)$ as a subgroup of $(L,+,0,-)^k$ of polynomial size. By then running Gaussian elimination (generalized to finite Abelian groups), or simply applying Theorem \ref{theorem:SMPsupernilpotent} one can then compute a compact representation of $\mathcal C(\mathbf u_1,\ldots,\mathbf u_n)$.

\section{The subpower membership problem of wreath products} \label{sect:reduction}
In this section we discuss our main methodological results. We show that, in some cases the subpower membership problem $\SMP(\algL \otimes \algU)$ of a wreath product can be reduced to $\CompRep(\algL \times \algU)$ and $\CompRep(\mathcal C)$.  We first show how such a reduction can be achieved relatively easily in the case where $\Clo(\algL \otimes \algU)$ contains $\Clo(\algL \times \algU)$ (i.e. the identity map is a retraction of the clone homomorphism from Observation \ref{observation:companion}):

\begin{theorem} \label{theorem:directprod}
Let $\algA = \algL \otimes^{(0,T)} \algU$ be a finite Mal'tsev algebra, and let $\mathcal C = \Diff_0(\algA)$. Further assume that $\Clo(\algL \times \algU) \subseteq \Clo(\algA)$. Then $\CompRep(\algA)$ (and hence also $\SMP(\algA)$) reduces in polynomial time to $\CompRep(\algL\times\algU)$ and $\CompRep(\mathcal C)$. 
\end{theorem}

\begin{proof}
Let $\mathbf a_1,\ldots,\mathbf a_n \in A^k$ an instance of $\CompRep(\algA)$; our goal is to find a compact representation of $\algB = \Sg_{\algA^k}(\mathbf a_1,\ldots,\mathbf a_n)$.  Let us write $\mathbf l_i$ and $\mathbf u_i$ for the projection of $\mathbf a_i$ to $L^k$ and $U^k$ respectively.  Let us further define $\algB^+ = \Sg_{(\algL\times \algU)^k}(\mathbf a_1,\ldots,\mathbf a_n)$.
Then
\begin{align*}
\algB &= \{(t^{\algL}(\mathbf l_1,\ldots, \mathbf l_n) + \hat t(\mathbf u_1,\ldots, \mathbf u_n), t^{\algU}(\mathbf u_1,\ldots, \mathbf u_n) \mid t \text{ is }F\text{-term} \}, \text{ and }\\
\algB^+ &= \{(t^{\algL}(\mathbf l_1,\ldots, \mathbf l_n),t^{\algU}(\mathbf u_1,\ldots, \mathbf u_n) \mid t \text{ is }F\text{-term}\}.
\end{align*}
Since $\Clo(\algL \times \algU) \subseteq \Clo(\algA)$, we can pick a Mal'tsev term of $\algA$ that is of the form $m^{\algA}((l_1,u_1),(l_2,u_2),(l_3,u_3)) = (l_1-l_2+l_3,m^{\algU}(u_1,u_2,u_3))$. Moreover, by Lemma \ref{lemma:difference}, every term $t^{\algA} \in \Clo(\algA)$ can be uniquely written as the sum of $t^{\algL \times \algU}$ (which by assumption is also in $\Clo(\algA)$) and some $\hat t \in \mathcal C$. Thus, every element of $\algB$ is equal to the sum of an element of $\algB^+$ and an expression $\hat t(\mathbf u_1,\ldots,\mathbf u_n)$.

Let $C^+$ be a compact representation of $\algB^+$, and $\hat C$ a compact representation of $\mathcal C(\mathbf u_1,\ldots,\mathbf u_n)$. Then, it follows that every tuple in $\algB$ can be written as
\begin{align} \label{fml:red1}
m(\ldots, m(\mathbf c_1,\mathbf d_2,\mathbf c_2), \ldots \mathbf d_n,\mathbf c_n) + \hat{\mathbf r}_1 - \hat{\mathbf s}_2 + \hat{\mathbf r}_2 - \ldots  - \hat{\mathbf s}_n + \hat{\mathbf r}_n,
\end{align}
for $\mathbf c_i, \mathbf d_i \in C^+$ and $\hat {\mathbf r}_i, \hat{\mathbf s}_i \in \hat C$. (We are aware that tuples in $C^+$ an $\hat C$ have different domains; here we follow the same convention as in Notation \ref{notation:diff}). Moreover, in formula (\ref{fml:red1}), any pair $\mathbf c_i, \mathbf d_i$ (respectively $\hat{\mathbf r}_i, \hat{\mathbf s}_i$) witnesses a fork in the $i$-th coordinate. By our choice of $m$ it is easy to see that formula (\ref{fml:red1}) can be rewritten to 
\begin{align*}
m(\ldots, m(\mathbf c_1 + \hat{\mathbf r}_1,\mathbf d_2 + \hat{\mathbf s}_2,\mathbf c_2 + \hat{\mathbf r}_2), \ldots \mathbf d_n + \hat{\mathbf s}_n,\mathbf c_n + \hat{\mathbf r}_n),
\end{align*}

Thus the elements $\mathbf c_i + \hat{\mathbf r}_i, \mathbf d_ + \hat{\mathbf s}_i$ witness forks of $\algB$ in the $i$-th coordinate. If we define $D = \{\mathbf c +\hat{\mathbf r} \mid \mathbf c \in C,\hat{\mathbf r} \in \hat C \}$, then it follows that $\Sig(D) = \Sig(\mathbf B)$. Moreover $D \subset \mathbf B$, and it is of polynomial size in $n$ and $k$, as $|D| \leq |C|\cdot|\hat C|$. Thus $D$ can be thinned out in polynomial time to a compact representation of $\algB$, which finishes our proof.
\end{proof}

We remark that, by following the proof of Theorem \ref{theorem:directprod}, also finding \emph{enumerated} compact representations in $\algA$ can be reduced to finding \emph{enumerated} compact representations in $\algL\times \algU$ and $\mathcal C$ (if $\mathcal C$ is given by some finite set of operations that generate it as a clonoid).

Unfortunately, the conditions of Theorem \ref{theorem:directprod} are not met for general wreath-products, not even if both $\algU$ and $\algL$ are both affine (the dihedral group $D_4$ can be shown to be a counterexample). But, if $\algU$ is supernilpotent, then we are able to prove the following reduction, independent of the conditions of Theorem \ref{theorem:directprod}:

\begin{theorem} \label{theorem:SMP}
Let $\algA = \algL \otimes \algU$ be a finite Mal'tsev algebra, and let $\mathcal C = \Diff_0(\algA)$ for some $0\in A$. Further, assume that $\algU$ is supernilpotent. Then $\SMP(\algA)$ reduces in polynomial time to $\CompRep(\mathcal C)$.
\end{theorem}

\begin{proof}
Let $\mathbf a_1,\ldots,\mathbf a_n , \mathbf b \in A^k$ an instance of $\SMP(\algA)$; our goal is to check whether $\mathbf b \in \algB = \Sg_{\algA^k}(\mathbf a_1,\ldots,\mathbf a_n)$. Let us write $\mathbf l_i$ and $\mathbf u_i$ for the projection of $\mathbf a_i$ to $L^k$ and $U^k$ respectively, and $\mathbf l_b$ and $\mathbf u_b$  for the projections of $\mathbf b$ to $L^k$ and $U^k$. Let $F$ be the signature of $\algA$ and $\algL\times \algU$, and let $\algB^+ = \Sg_{(\algL\times \algU)^k}(\mathbf a_1,\ldots,\mathbf a_n)$. Then
\begin{align*}
\algB &= \{(t^{\algL}(\mathbf l_1,\ldots, \mathbf l_n) + \hat t(\mathbf u_1,\ldots, \mathbf u_n), t^{\algU}(\mathbf u_1,\ldots, \mathbf u_n) \mid t \text{ is }F\text{-term} \}, \text{ and }\\
\algB^+ &= \{(t^{\algL}(\mathbf l_1,\ldots, \mathbf l_n),t^{\algU}(\mathbf u_1,\ldots, \mathbf u_n) \mid t \text{ is }F\text{-term}\}.
\end{align*}

Recall the definition of $t^{\algA} \sim s^{\algA}$ from Definition \ref{definition:diffclonoid}. If $T$ is a $\sim$-transversal set of $\{ t^{\algA} \in \Clo(\algA) \mid t^{\algU}(\mathbf u_1,\ldots, \mathbf u_n) = \mathbf u_b\}$, then clearly $\mathbf b \in B$ iff $\exists t \in T$ and $\mathbf d \in \mathcal C(\mathbf u_1,\ldots,\mathbf u_n)$, with $\mathbf b = t(\mathbf a_1,\ldots,\mathbf a_n) + \mathbf d$. So, intuitively speaking, the goal of this proof is to first compute such a transversal set, by computing an enumerated compact representation of $\{ (\mathbf l,  \mathbf u) \in \algB^+ \mid \mathbf u = \mathbf u_b \}$ and then use it together with a compact representation of $\mathcal C(\mathbf u_1,\ldots, \mathbf u_n)$ to check membership of $\mathbf b$ in $\algB$. 

In practice we need however to consider a relation of higher arity than $\algB^+$, since term operations of $\algL \times \algU$ are not uniquely determined by their value on $\mathbf a_1,\ldots,\mathbf a_n$. So let $S$ be the degree of supernilpotence of $\algU$ (and hence also $\algL \times \algU$). If we think about $\mathbf a_1,\ldots,\mathbf a_n$ as the columns of a matrix of dimension $k\times n$, then let $\tilde{\mathbf a}_1,\ldots,\tilde{\mathbf a}_n \in A^l$ be its extension by rows that enumerate $H = \{(a_1,\ldots,a_n) \in A^n \mid |\{i \colon a_i \neq 0 \}| \leq S \}$ (hence $l \leq k + |A|^S \binom{n}{S}$).

It follows from Theorem \ref{theorem:SMPsupernilpotent} that we can compute an enumerated compact representation $\tilde{C}$ of $\Sg_{(\algL\times \algU)^l}(\tilde{\mathbf a}_1,\ldots,\tilde{\mathbf a}_n)$ in polynomial time in $n$ and $l$. So, every element in $\tilde{B} = \Sg_{(\algL\times \algU)^l}(\tilde{\mathbf a}_1,\ldots,\tilde{\mathbf a}_n)$ can be written as $m(\ldots m(\tilde{\mathbf c}_1, \tilde{\mathbf d}_2, \tilde{\mathbf c}_2) \ldots \tilde{\mathbf d}_l,\tilde{\mathbf c}_l)$, for $(\tilde{\mathbf c}_i,p_{\tilde{\mathbf c_i}}) , (\tilde{\mathbf d}_i,p_{\tilde{\mathbf d_i}}) \in \tilde{C}$, where $\tilde{C}$ is of size at most $2l|A|^2$,  and every element of $\tilde{\mathbf c} \in \tilde{C}$ is equal to $p_{\tilde{\mathbf c}}(\tilde{\mathbf a}_1,\ldots,\tilde{\mathbf a}_n) = \tilde{\mathbf c}$ for the given circuit $p_{\tilde{\mathbf c}}$ of polynomial size.

By Theorem \ref{theorem:supernilpotent}, in an $S$-supernilpotent algebra, every term operation is already completely determined by its values on the subset $H$. It follows, that every $n$-ary term operation of $\algL\times\algU$ can be uniquely described by a circuit $m(\ldots m(p_{\tilde{\mathbf c}_1},p_{\tilde{\mathbf d}_2},p_{\tilde{\mathbf c}_2}), \ldots p_{\tilde{\mathbf d}_l},p_{\tilde{\mathbf c}_l})$ for $\tilde{\mathbf c}_i, \tilde{\mathbf d}_i \in \tilde{C}$.  By definition of $\sim$, it follows that also every $n$-ary term operation of $\algA$ is $\sim$-equivalent to the operation given by the circuit described by a circuit $m(\ldots m(p_{\tilde{\mathbf c}_1},p_{\tilde{\mathbf d}_2},p_{\tilde{\mathbf c}_2}), \ldots p_{\tilde{\mathbf d}_l},p_{\tilde{\mathbf c}_l})$ for $\tilde{\mathbf c}_i, \tilde{\mathbf d}_i \in \tilde{C}$.

We are however only interested in terms $t$ such that $t^{\algU}$ maps $\mathbf u_1,\ldots, \mathbf u_n$ to the value $\mathbf u_b$. By Lemma \ref{lemma:fixvalues}, we can also compute an enumerated compact representation $\tilde{C}'$ of $\{ (\tilde{\mathbf l},  \tilde{\mathbf u}) \in \Sg_{\algL\times \algU}(\tilde{\mathbf a}_1,\ldots,\tilde{\mathbf a}_n) \mid \tilde{\mathbf u}(i) = \mathbf u_b(i) $ for all  $ i =1,\ldots,k\}$ in polynomial time. (Although we only prove Lemma \ref{lemma:fixvalues} for fixing variables to constants, we remark that it can straightforwardly be generalized to fixing the value of the variables to domains $L\times \{u\}$. Alternatively, this can also be achieved by regarding $\Sg_{(\algL\times \algU)^l}(\tilde{\mathbf a}_1,\ldots,\tilde{\mathbf a}_n)$ as a subalgebra of $\algU^l\times\algL^l$, which however would require us to work with relations on different domains).

If $\tilde{C}' = \emptyset$, then we output ``False'', as then $\mathbf u_b \notin \Sg_{\algU^k}(\mathbf u_1,\ldots, \mathbf u_n)$. Otherwise, let $C = \{ p_{\tilde{\mathbf c}}^{\algA}(\mathbf a_1,\ldots,\mathbf a_n) \mid \tilde{\mathbf c} \in \tilde{C}' \}$.  Also, let $\hat C$ be a compact representation of $\mathcal C(\mathbf u_1,\ldots, \mathbf u_n)$.  By our proof, every element of $\{ (\mathbf l,  \mathbf u) \in \algB \mid \mathbf u = \mathbf u_b \}$ is equal to the sum of an element $m^{\algA}(\ldots, m^{\algA}(\mathbf c_1,\mathbf d_2,\mathbf c_2), \ldots \mathbf d_n,\mathbf c_n)$ with $\mathbf c_i,\mathbf d_i \in C$ and an element of $\mathcal C(\mathbf u_1,\ldots, \mathbf u_n)$. Since $m$ is an affine Malt'sev operation when restricted to $\{ (\mathbf l,  \mathbf u) \in \algB \mid \mathbf u = \mathbf u_b \}$ this means that $\mathbf b \in \algB$ iff $\mathbf l_b$ is in the affine closure of all elements $\mathbf c + \hat{\mathbf r}$ with $\mathbf c \in C$ and $\hat{\mathbf r}\in \hat C$. But this can be checked in polynomial time (by generalized Gaussian elimination, or Theorem \ref{theorem:SMPsupernilpotent}), which finishes the proof.
\end{proof}

\section{Clonoids between affine algebras} \label{section:clonoids}
We continue our paper with an analysis of clonoids between affine algebras to prove our main result, Theorem \ref{theorem:main}.

For a prime $p$, let us write $\Z_p$ for the cyclic group of order $p$, i.e. $\Z_p = (\{0,1,\ldots,p-1\},+,0,-)$. Let us further define the idempotent reduct $\Z_p^{id} = (\{0,1,\ldots,p-1\}, x-y+z)$.  Using the unary terms $a x = x + \cdots + x$ ($a$-times), for $a \in \Z_p$, we can regard $\Z_p$ as a vector space over the $p$-element field. More general, using this notation, we will also consider finite Abelian groups $(L,+,0,-)$ as modules over $\Z_{|L|}$.

For short, we are going to denote constant $1$-tuples by $\mathbf 1 = (1,1,\ldots,1) \in \Z_p^n$. For two vectors $\mathbf{a}, \mathbf x \in \Z_p^n$, we further denote by $\mathbf{a} \cdot \mathbf x = \sum_{i = 1}^n \mathbf{a}(i) \cdot \mathbf x(i)$ the standard inner product. Then $\Clo(\Z_p) = \{\mathbf x \mapsto \mathbf a \cdot \mathbf x \mid \mathbf a \in \Z_p^n \}$ and $\Clo(\Z_p^{id}) = \{ \mathbf x \mapsto \mathbf a \cdot \mathbf x \mid \mathbf a \in \Z_p^n, \mathbf{a} \cdot \mathbf 1 = 1 \}$.

In this section, we are going to study clonoids between affine algebras $\algU$ and $\algL$, such that $|U| = p$ for some prime $p$, and $p \nmid |L|$. Since every such affine algebra $\algU$ has $x-y+z$ as a term operation, it makes sense to study the special case $\algU = \Z_p^{id}$. As we are in particular interested in difference clonoids, we furthermore can assume that $\algL$ contains a constant operations $0$ (see Lemma \ref{lemma:difference}), and hence the operations of the Abelian group $(L,+,0,-)$. We remark that our analysis is structurally similar to (but not covered by) Fioravanti's classification of $(\Z_p,\Z_q)$-clonoids \cite{fioravanti-clonoids}.

\subsection{$(\Z_p^{id}, \algL)$-clonoids satisfying $p \nmid |L|$ and $f(x,x,\ldots,x) = 0$}
Throughout this subsection, let $p$ be a prime, and $\algL = (L,+,0,-)$ an Abelian group with $p \nmid |L|$, and $\mathcal C$ be a $(\Z_p^{id}, \algL)$-clonoids satisfying $f(x,x,\ldots,x) = 0$ for all $f \in \mathcal C$ and $x \in \Z_p$. In other words, for every $n\in \N$, $\mathcal C$ maps all tuples from the diagonal $\Delta^n = \{(x,x\ldots,x) \in \Z_p^n\}$ to $0$. We are going to prove that $\mathcal C$ is generated by its binary elements, and therefore by any set of generators $B$ of $\mathcal C^{(2)} \leq \algL^{\Z_p^2}$. Moreover, from $B$, we are going to construct canonical generating set of the $n$-ary functions $\mathcal C^{(n)} \leq \algL^{\Z_p^n}$. 
We are, in particular going to use the following set of cofficient vectors for every $n>2$:
\[ C_n = \{ \mathbf a \in \Z_p^n \mid \exists i > 1 \colon \mathbf a(1) = \mathbf a(2) = \ldots = \mathbf a(i-1)  = 0, \mathbf a(i) = 1 \}. \]

\begin{observation}
Every 2-dimensional subspace $V\leq \Z_p^n$ containing the diagonal $\Delta^n$ has a unique parameterization by the map 
\begin{align*}
e_{\mathbf c}(x,y) &= x (\mathbf 1 - \mathbf c) + y \mathbf c = (x, \mathbf c(2)x + (1-\mathbf c(2))y,\ldots,\mathbf c(n)x + (1-\mathbf c(n))y),
\end{align*}
for some $\mathbf c \in C_n$.
\end{observation}

\begin{proof}To see this, note that $V$ contains $\mathbf 1$, and can be therefore parameterized by $e_{\mathbf d}(x,y)$, for some $\mathbf d \notin \Delta^n$. So there is an index $i$ with $\mathbf d(1) = \ldots = \mathbf d(i-1) \neq \mathbf d(i)$. If $\mathbf d \notin C_n$, then we define $\mathbf c = (\mathbf d(i)-\mathbf d(1))^{-1}(\mathbf d - \mathbf d(1) \mathbf 1)$; clearly $\mathbf c \in C_n$, and $\mathbf c$ and $\mathbf 1$ still generate $V$. It is further not hard to see that different elements of $C_n$ generate different planes together with $\mathbf 1$, thus we obtain a unique parameterization of $V$ by $e_{\mathbf c}(x,y)$.
\end{proof}

\begin{lemma} \label{lemma:binary1}
Let $f \in \mathcal C^{(2)}$. Then, there is a function $f_n \in \mathcal C^{(n)}$, such that

$$f_n(x_1,x_2,\ldots,x_n) = \begin{cases}
f(x_1,x_2) \text{ if } x_2 = x_3 = \ldots = x_n\\
0 \text{ else.}
\end{cases}
$$
\end{lemma}

\begin{proof}
We prove the lemma by induction on $n$. For $n=2$, we simply set $f_2 = f$. For an induction step $n \to n+1$, we first define $t_{n+1}(x_1,x_2,\ldots,x_n,x_{n+1})$ as the sum
\begin{align*} \sum_{\mathbf a \in \Z_p^{n-1}} f_n(x_1,x_2 + \mathbf a(1)(x_{n+1}-x_n),\ldots,x_n+\mathbf a(n-1)(x_{n+1}-x_n)) \\
-\sum_{\mathbf a \in \Z_p^{n-1}} f_n(x_1,x_1 + \mathbf a(1)(x_{n+1}-x_n),\ldots,x_1+\mathbf a(n-1)(x_{n+1}-x_n)).
\end{align*}
Note that, if $x_{n+1} \neq x_n$, then $t_{n+1}$ evaluates to $\sum_{\mathbf a \in \Z_p^{n-1}}f(x_1,\mathbf a) - \sum_{\mathbf a \in \Z_p^{n-1}}f(x_1,\mathbf a) = 0$. On the other hand, if $x_n = x_{n+1}$, then the second sum is equal to $0$, while the first one is equal to $p^{n-1} f_n(x_1,x_2,\ldots,x_n)$. By the induction hypothesis, the function $f_{n+1} = p^{-(n-1)} t_{n+1}$ satisfies the statement of the lemma (note that $p^{-(n-1)}$ exist modulo $|L|$, since $p \nmid |L|$).
\end{proof}

We can prove an analogue statement for all 2-dimensional subspaces of $\Z_p^n$ containing $\Delta^n$:

\begin{lemma} \label{lemma:binary2}
Let $f \in \mathcal C^{(2)}$, and $\mathbf c \in C_n$. Then there is a function $f^{\mathbf c} \in \mathcal C^{(n)}$, such that
$$f^{\mathbf c}(x_1,x_2,\ldots,x_n) = \begin{cases}
f(x,y) \text{ if } (x_1,x_2,\ldots,x_n) = e_{\mathbf c}(x, y) \\
0 \text{ else.}
\end{cases}
$$
\end{lemma}

\begin{proof}
Let $\mathbf c \in C^{(n)}$. There is a matrix $\mathbf T \in \Z_p^{n\times n}$, such that $\mathbf T \cdot \mathbf 1 = \mathbf 1$ and $\mathbf T \cdot (\mathbf 1 -\mathbf c) = \mathbf e_1$.
Let $f_n$ as in Lemma 2 and $f^{\mathbf c} := f_n\circ T$. Note that by the first condition, all rows of $T$ sum up to $1$, hence $T$ can be expressed by terms of $\Z_p^{id}$. Then $f^{\mathbf c}(e_{\mathbf c}(x,y)) = f_n(T(x(\mathbf 1 -\mathbf c) + y\mathbf c)) = f_n(x \mathbf e_1 + y(\mathbf 1-\mathbf e_1)) = f(x,y)$, and $f^{\mathbf c}(\mathbf x) = 0$ for $\mathbf x \notin e_{\mathbf c}(\Z_p^2)$.
\end{proof}

We are now ready to prove the main result of this section:

\begin{lemma} \label{lemma:idempotent}
Let $\mathcal C$ be a $(\Z_p^{id}, \algL)$-clonoid satisfying $\forall f \in \mathcal C, x \in \Z_p \colon f(x,\ldots,x) = 0$, and let $B$ be a generating set of $\mathcal C^{(2)} \leq \algL^{\Z_p^2}$. Then 
\begin{bracketenumerate}
\item $\mathcal C$ is the $(\Z_p^{id}, \algL)$-clonoid generated by $B$, and
\item $B_{n} := \{ f^{\mathbf c} \mid f \in B, \mathbf c \in C_n \}$ is a generating set of $\mathcal C^{(n)}$ in $\algL^{\Z_p^n}$,
\end{bracketenumerate}
\end{lemma}

\begin{proof}
For any $g \in \mathcal C^{(n)}$ and $\mathbf c \in C^{(n)}$, let us define the binary operation $g_{\mathbf c} = f(e_{\mathbf c}(x,y)) \in \mathcal C^{(2)}$. By Lemma \ref{lemma:binary2}, $g_{\mathbf c}$ generates a function $g_{\mathbf c}^{\mathbf c} \in \mathcal C^{(n)}$, that agrees with $f(x,y)$ on all tuples of the form $e_{\mathbf c}(x,y)$, and that is $0$ else. Since every point of $\Z_p^n \setminus \Delta^n$ is in the image of a unique map $e_{\mathbf c}$, we get $g = \sum_{\mathbf c \in C_n} g_{\mathbf c}^{\mathbf c}$. Every element of the form $g^{\mathbf c}$ can be clearly written as a linear combination of elements $f^{\mathbf c}$, where $f \in B$. It follows that $B_{n}$ generates $\mathcal C^{(n)}$ in $\algL^{\Z_p^n}$, and that the clonoid generated by $B$ is $\mathcal C$.
\end{proof}

We remark that if $\algL = \Z_q$ for a prime $q\neq p$, and $B$ is a basis of the vector space $\mathcal C^{(2)} \leq \algL^{\Z_p^2}$, then also $B_n$ is a basis. The generating set $B_{n}$ can be used to decide efficiently the following version of the subpower membership problem for $\mathcal C$:

\begin{lemma} \label{lemma:SMP}
Let $\mathcal C$ be a $(\Z_p^{id},\algL)$-clonoid satisfying $\forall f \in \mathcal C, x \in \Z_p \colon f(x,\ldots,x) = 0$. Then we can solve $\CompRep(\mathcal C)$ in polynomial time.
\end{lemma}

\begin{proof}
By Lemma \ref{lemma:idempotent}, $\mathcal C^{(n)}$ is the linear closure of $B_{n}$.  Thus $\mathcal C(\mathbf u_1,\ldots, \mathbf u_n)$ is equal to the linear closure of $B_{n}(\mathbf u_1,\ldots, \mathbf u_n) := \{ f^{\mathbf c}(\mathbf u_1,\ldots, \mathbf u_n) \mid f \in B, \mathbf c \in C_n \}$.

Note that the $i$-th entry $f^{\mathbf c}(\mathbf u_1,\ldots, \mathbf u_n)(i)$ of such a generating element can only be different from $0$ if $(\mathbf u_1,\ldots, \mathbf u_n)(i)$ lies in the 2-dimensional subspace generated by the diagonal $\Delta^n$ and $\mathbf c$. Thus, there are at most $k$ many vectors $\mathbf c \in C_n$ such that $f^{\mathbf c}(\mathbf u_1,\ldots, \mathbf u_n) \neq \mathbf 0$, let $\mathbf c_1,\ldots,\mathbf c_l$ be an enumeration of them. Clearly $D= \{ f^{\mathbf c}(\mathbf u_1,\ldots, \mathbf u_n) \mid f \in B, \mathbf c \in \{\mathbf c_1,\ldots,\mathbf c_l\} \}$ generates $\mathcal C(\mathbf u_1,\ldots, \mathbf u_n)$; note that we can compute it in linear time $O(kn)$. From the generating set $D$ we can compute a compact representation of  $\mathcal C(\mathbf u_1,\ldots, \mathbf u_n)$ in polynomial time (by generalized Gaussian elimination, or Theorem \ref{theorem:directprod}).
\end{proof}

\subsection{General $(\Z_p^{id}, \algL)$-clonoids}
For an arbitrary $(\Z_p^{id}, \algL)$-clonoid $\mathcal C$, let us define the subclonoid $\mathcal C^{\Delta} = \{f\in \mathcal C \colon f(x,\ldots,x) = 0 \}$. We then show, that every $f \in \mathcal C$ can be written in a unique way as the sum of an element of $\mathcal C^{\Delta}$, and a function that is generated by $\mathcal C^{(1)}$. For this, we need the following lemma:

\begin{lemma} \label{lemma:diagonal}
For any $f \in \mathcal C^{(n)}$, let us define
$$f'(\mathbf x) = p^{(1-n)}\sum_{\substack{\mathbf a \in \Z_p^n\\ \mathbf a \cdot \mathbf 1 = 1}} f(\mathbf a \cdot \mathbf x,\mathbf a \cdot \mathbf x,\ldots,\mathbf a \cdot \mathbf x).$$
Then $f-f' \in \mathcal C^{\Delta}$, and $f'$ is generated by $\mathcal C^{(1)}$.
\end{lemma}

\begin{proof}
By definition, $f'$ is in the clonoid generated by the unary function $f(x,x,\ldots,x) \in \mathcal C^{(1)}$. 
Thus, to prove the lemma, it is only left to show that $f-f' \in \mathcal C^{\Delta}$, or, in other words, that $f(\mathbf x) = f'(\mathbf x)$ for $\mathbf x \in \Delta$.

But this is not hard to see, since 
$$f'(x,x,\ldots,x) = p^{(1-n)}\sum_{\substack{\mathbf a \in \Z_p^n\\ \mathbf a \cdot \mathbf 1 = 1}} f(x,x,\ldots,x) = f(x,x,\ldots,x).$$
\end{proof}

It follows in particular from Lemma \ref{lemma:diagonal} and Lemma \ref{lemma:idempotent} that every $(\Z_p^{id}, \algL)$-clonoid $\mathcal C$ is generated by any set $A \cup B$, such that $A$ generates $\mathcal C^{(1)}$ in $\algL^\Z_p$ and $B$ generates $\mathcal C_\Delta^{(2)}$ in $\algL^{\Z_p^2}$. Note that the clonoid generated by $A$ does not need to be disjoint from $\mathcal C^{\Delta}$. We can, however, still prove results analogous to the previous section.

\begin{lemma} \label{lemma:linbasisgeneral}
Let $\mathcal C$ be a $(\Z_p^{id},\algL)$-clonoid, let $A$ be a generating set of $\mathcal C^{(1)} \leq \algL^{\Z_p}$ and $B$ a generating set of $\mathcal C_\Delta^{(2)} \leq \algL^{\Z_p^2}$. For every $n$, let us define $A_n = \{ \sum_{\mathbf a \in \Z_p^n, \mathbf a \cdot \mathbf 1 = 1} f(\mathbf a \cdot \mathbf x) \mid f \in A \}$ and let $B_{n}$ be defined as in Lemma \ref{lemma:idempotent}. Then
$A_n \cup B_{n}$ is a generating set of $\mathcal C^{(n)}$ in $\algL^{\Z_p^n}$.
\end{lemma}

\begin{proof}
We already know from Lemma \ref{lemma:idempotent} that $B_{n}$, generates $\mathcal C_\Delta^{(n)} \leq \algL^{\Z_p^n}$. 

By Lemma \ref{lemma:diagonal}, every element $f\in \mathcal C^{(n)}$ can be uniquely written as the sum $f'$ and $f-f'$. Furthermore $f'$, by definition, is generated by $A_n$, and $f-f'$ is in  $\mathcal C_\Delta^{(n)}$, which finishes our prove.
\end{proof}

Lemma \ref{lemma:linbasisgeneral} allows us to straightforwardly generalize Lemma \ref{lemma:SMP} to arbitrary $(\Z_p^{id},\algL)$-clonoids:

\begin{lemma} \label{lemma:SMP2}
Let $\mathcal C$ be a $(\Z_p^{id},\algL)$-clonoid. Then $\CompRep(\mathcal C) \in \comP$.
\end{lemma}

\begin{proof}
Let $A_n$ and $B_n$ be defined as in Lemma \ref{lemma:linbasisgeneral}. Our goal is to compute a compact representation of $\mathcal C(\mathbf u_1,\ldots, \mathbf u_n)$ for some given $\mathbf u_1,\ldots, \mathbf u_n \in \Z_p^k$. By Lemma \ref{lemma:linbasisgeneral}, every $g \in \mathcal C$ decomposes into the sum of $g'$ and $g - g'$, where $g'$ is generated by $A_n$ and $g - g'$ is generated by $B_{n}$. Thus any image $g(\mathbf u_1,\ldots, \mathbf u_n)$ is in the linear closure of all tuples $f(\mathbf u_1,\ldots, \mathbf u_n)$, for $f \in A_n$ and $B_{n}(\mathbf u_1,\ldots, \mathbf u_n) = \{f^{\mathbb c}(\mathbf u_1,\ldots, \mathbf u_n) \mid f \in B, \mathbb c \in C_n\}$ in $\algL^k$. There are at most $|A|$-many tuples of the first form. Furthermore, as in the proof of Lemma \ref{lemma:SMP} we can compute a generating set of $B_{n}(\mathbf u_1,\ldots, \mathbf u_n)$ in polynomial time. By generalized Gaussian elimination (or Theorem \ref{theorem:SMPsupernilpotent}), we can obtain a compact representation from these generators in polynomial time.
\end{proof}

Lemma \ref{lemma:SMP2} allows us to finish the proof of our main result:

\begin{theorem} \label{theorem:main}
Let $\algA$ be a finite Mal'tsev algebra, with a central series $0_\algA < \rho < 1_\algA$ such that $|\algA/\rho| = p$ is a prime, and the blocks of $\rho$ are of size coprime to $p$. Then $\SMP(\algA) \in \comP$.
\end{theorem}

\begin{proof}
By Theorem \ref{theorem:wreath}, $\algA$ is isomorphic to a wreath product $\algL \otimes \algU$,  such that $\algU$, $\algL$ are affine with $|U| = p$ and $|L|$ coprime to $p$.  By Theorem \ref{theorem:SMP}, $\SMP(\algA)$ reduces to $\CompRep(\Diff_0(\algA))$ in polynomial time. The difference clonoid is a clonoid from $\algU$ to $(\algL,0)$. Since both $\algL$ and $\algU$ are affine, and therefore have term operations describing $x-y+z$, $\Diff_0(\algA)$ is also a clonoid from $\Z_p^{id}$ to $(L,+,0,-)$. By Lemma \ref{lemma:SMP2}, $\CompRep(\Diff_0(\algA))$ is solvable in polynomial time, which finishes the proof.
\end{proof}

\begin{corollary} \label{cor:main}
For every nilpotent Mal'tsev algebra $\algA$ with $|A| = pq$ for distinct primes $p \neq q$, we have $\SMP(\algA) \in \comP$.
\end{corollary}

\begin{proof}
If $\algA$ is affine, then the result holds by (generalized) Gaussian elimination. So assume that $\algA$ is 2-nilpotent, but not affine. So $\algA$ is isomorphic to $\algL \otimes \algU$, and wlog. $|L|=q$ and $|U|=p$. Then the result follows directly from Theorem \ref{theorem:main}.
\end{proof}

\section{Discussion} \label{sect:discussion}

In Theorem \ref{theorem:main} we proved that every Mal'tsev algebra, which can be written as a wreath product $\algL \otimes \algU$ with $|U|=p$ and $p \nmid |L|$ has a tractable subpower membership problem.  But, since the reduction discussed in Theorem \ref{theorem:SMP} extends beyond this case, it is natural to ask, whether we can also extend the tractability also extends to all those cases:

\begin{question} \label{question:tractable}
Is $\SMP(\algL \otimes \algU) \in \comP$ for every supernilpotent Mal'tsev algebra $\algU$?
\end{question}

In particular, if $\algU$ is affine, Question \ref{question:tractable} asks, whether the subpower membership problem of all finite 2-nilpotent Mal'tsev algebras can be solved in polynomial time. By Theorem \ref{theorem:SMP}, this reduces to computing compact representations with respect the clonoids between affine algebras. Thus answering the question requires a better understanding of such clonoids.

Very recent result \cite{MW-clonoidsmodules} study such clonoids in the case where $\algU$ has a distributive congruence lattice, and $\algL$ is coprime to $\algU$. Such clonoids are always generated by functions of bounded arity (as in Lemma \ref{lemma:binary2}), thus we expect then similar argument as in Lemma \ref{lemma:SMP2} to work in solving $\CompRep(\mathcal C)$. We remark that the fact that every \emph{full} clonoid between such $\algU$ and $\algL$ is finitely generated was already implicitly used in \cite{KKK-CEQV-2nilpotent} to obtain polynomial time algorithm for checking whether two circuits over a 2-nilpotent algebra are equivalent. However \cite{MW-clonoidsmodules} does not cover all clonoids between affine algebras; e.g. for the case $\algU = \Z_p \times \Z_p$ and coprime $\algL$ nothing is known so far.

A reason why much emphasis is placed on coprime $\algU$ and $\algL$ is, that their wreath products $\algL \otimes^{0,T} \algU$ are not supernilpotent (for non-trivial operations $T$), and therefore not covered by Theorem \ref{theorem:SMPsupernilpotent}. In fact, finite Mal'tsev algebras in finite language are supernilpotent if and only if they decompose into the direct product of nilpotent algebras of prime power size (see e.g. \cite[Lemma 7.6.]{AM-commutator}).  It is further still consistent with our current knowledge that the conditions of Theorem \ref{theorem:directprod} are always met, for coprime  $\algL$ and $\algU$. This naturally leads to the question:

\begin{question} \label{question:product}
Is $\Clo(\algL \times \algU) \subseteq \Clo(\algL \otimes \algU)$, for all finite nilpotent Mal'tsev algebras $\algL \otimes \algU$ where $\algL$ and $\algU$ of coprime size?
\end{question}

In fact, in an unpublished proof \cite{AMW-reduct}, a positive answer to Question \ref{question:product} is given in the case that $\Clo(\algL \otimes \algU)$ contains a constant operations. A more general version of Question \ref{question:product} would ask, whether every finite nilpotent Mal'tsev algebra $\algA$ has a Mal'tsev term $m$, such that $(A,m)$ is supernilpotent.

Lastly we would like to mention that recently the property of \emph{short pp-defitions} was suggested as a witnesses for $\SMP(\algA) \in \comcoNP$. While Mal'tsev algebras that generate residually finite varieties have short pp-definitions \cite{BK-shortpp}, it is not know whether this is true in the nilpotent case. Thus we ask:

\begin{question} \label{question:shortpp}
Does every finite nilpotent Mal'tsev algebras $\algA$ have short pp-definitions (and hence $\SMP(\algA) \in \comNP \cap \comcoNP$)?
\end{question}

Studying Question \ref{question:shortpp} might especially be a useful approach to discuss the complexity for algebras of high nilpotent degree, if studying the corresponding difference clonoids turns out to be too difficult or technical endeavor.

\bibliography{bibliography}

\appendix

\section{Proof of Lemma \ref{lemma:fixvalues}} \label{appendix:fixval}

In this appendix we prove the second statement of Lemma \ref{lemma:fixvalues}, i.e. we show that for a given \emph{enumerated} compact representation $R$ of a subpower $\mathbf R = \Sg_{\algA^k}(X)$ of some Mal'tsev algebra, we can obtain an \emph{enumerated} compact representation $R'$ of $\mathbf R \cap \{ \mathbf x \in A^k \mid \mathbf x(1) = a_1,  \ldots, \mathbf x(k) = a_k\}$ for a given list of constants $a_1,\ldots,a_k$. In Algorithm \ref{alg:fixvalue} we describe the algorithm \texttt{Fix-value}$(R,a)$ that fixes the first coordinate of $\mathbf R$ to $a$; iterating this algorithm $m$ times results in the statement of the Lemma.

We remark that \texttt{Fix-value}$(R,a)$ is based on the \texttt{Fix-values} algorithm in \cite[Algorithm 5]{brady-CSPnotes}); although, for simplicity, we only fix the value of one coordinate.  Line 7 and 8 corresponds to the call of the subroutine \texttt{Nonempty} in  \cite[Algorithm 5]{brady-CSPnotes}), with the difference that we compute all the elements of the set $T_{j} = \{ (x,y) \in \proj_{1,j} \mathbf R \mid x = a \}$, instead of computing a witness for $(a,y) \in T_{j}$ once at a time.

We claim that the running time of \texttt{Fix-Value}$(R,a)$ is $O(|A|^2 \cdot n))$. For this note that the exhaustive search in line 7 and 8 will simply recursively apply $m$ to elements from $\proj_{1,j} (R)$ until the set is closed under $m$, and then select all values with $x_1=a$. Since $|\proj_{1,j}\mathbf R| \leq A^2$ this takes at most $|A|^2$ steps. For this reason also the size of the defining circuits $C_j$ (when e.g. stored all together as a circuit with multiple out-gates) is bounded by $|R| + |A|^2$.  Since the for-loop of line 6 has at most $n$ iterations, it follows that both the running time of the algorithm and the size of the defining circuits in $R'$ are bounded by $O(|A|^2 \cdot n)$.

If we then repeatedly call \texttt{Fix-value} to fix the value of the first $m$-many values of $\mathbf R$, this results in an algorithm that runs in time $O(|A|^2 \cdot nm))$.

Thus, the only thing that remains to prove is that the algorithm \texttt{Fix-Value} is correct. i.e. it indeed outputs an enumerated $R'$ with $\Sig(R') = \Sig(\mathbf R \cap \{ \mathbf x \in A^k \mid \mathbf x(1) = a \})$ (if the output is not empty).  So assume that $(i,b,c) \in \Sig(\mathbf R \cap \{ \mathbf x \in A^k \mid \mathbf x(1) = a \})$. If $i =1$, then clearly $(i,b,c) = (1,a,a)$, which is in $\Sig(R')$.  So let us assume wlog. that $i > 1$.  Since $R$ is a compact representation of $\mathbf R$, there exist tuples $\mathbf r_b, \mathbf r_c \in R$ (and defining circuits $p_{\mathbf r_b}$ and $p_{\mathbf r_c}$), witnessing that $(i,b,c) \in \Sig(R)$. Then $R'$ contains the tuples $\mathbf t$ and $\mathbf s = m(\mathbf t, \mathbf r_b, \mathbf r_c)$, as constructed in line 12 and 13 of Algorithm \ref{alg:fixvalue}. Since $\mathbf r_b$ and $\mathbf r_c$ agree on the first $i-1$ coordinates also $\mathbf t$ and $\mathbf s$ do. Moreover $\mathbf t(1) = a$,  $\mathbf t(i) = b$, and $\mathbf s(i) = m(b, b,c) = c$, thus $\mathbf t$ and $\mathbf s$ witness $(i,b,c) \in \Sig(\mathbf R \cap \mathbf x(1) = a)$. It follows that $\Sig(R') = \Sig(\mathbf R \cap \{ \mathbf x \in A^k \mid \mathbf x(1) = a \})$, which is what we wanted to prove.

\begin{algorithm}
\caption{An algorithm that, for a given enumerated compact representations $R$ of $\mathbf R = \Sg_{\algA^k}(X)$ outputs an enumerated compact representation $R'$ of the relation that fixes $x_1 = a$,  (where the defining circuits of $R'$ are evaluated on $X$)}
\label{alg:fixvalue}
\begin{algorithmic}[1]
\Procedure{Fix-Value}{$a \in A$,  $R$ (enum. c.r. of $\mathbf R = \Sg_{\algA^k}(X)$), Mal'tsev term $m$}
	\If {$(1,a,a) \not\in \Sig(R)$}
		\Return  $\emptyset$
	\Else
		\State Let $(\mathbf t, p_{\mathbf t}) \in R$ be such that $(\mathbf t, \mathbf t)$ is a witness of $(1,a,a) \in \Sig(R)$.
    	\State $R' = \{ (\mathbf t, p_{\mathbf t}) \}$
    		\For{$j > 1$}
    			\State Recursively apply $m$ to $\proj_{1,j}(R)$ to compute $T_{j} = \{ (x,y) \in \proj_{1,j} (\mathbf R) \mid x = a \}$,  
    			\State and circuits $C_{j} = \{ p_{(x,y)} \mid (x,y) \in T_j  \}$ such that $\proj_{1,j}(p_{(x,y)}(X)) = (x,y)$.
        	\For{$(j,b,c) \in \Sig(R)$}
        		\State Let $(\mathbf r_b, p_{\mathbf r_b}), (\mathbf r_c, p_{\mathbf r_c}) \in R$ be witnesses of $(j,b,c) \in \Sig(R)$
        		\If{$(a,b) \in T_j$}
        			\State Let $\mathbf t = p_{(a,b)}(X)$
        			\State $\mathbf s = m(\mathbf t, \mathbf r_b, \mathbf r_c)$ and $p_{\mathbf s} = m(p_{(a,b)}, p_{\mathbf r_b},p_{\mathbf r_c})$
        			\State $R'  = R' \cup \{ (\mathbf t, p_{(a,b)}), (\mathbf s,  p_{\mathbf s}) \}$
        		\EndIf
        	\EndFor
        	\EndFor
	\EndIf
\State \Return $R'$
\EndProcedure
\end{algorithmic}

\end{algorithm}

\end{document}